\documentclass[11pt]{amsart}

\usepackage{amsmath,amsthm,amssymb}
\usepackage{color}

\usepackage{enumitem}
\usepackage{float}
\usepackage[T1]{fontenc}
\usepackage{graphics}
\usepackage{fullpage}
\usepackage[colorlinks,linkcolor=blue]{hyperref}
\usepackage{mathtools}
\usepackage{MnSymbol}
\usepackage{multicol}
\usepackage{parskip}
\usepackage{comment}

\usepackage{tikz}
\usetikzlibrary{cd}
\usetikzlibrary{decorations.markings}
\usepackage{wrapfig}

\usepackage[colorinlistoftodos,prependcaption]{todonotes}
\usepackage[all]{xy}

\makeatletter
\DeclareRobustCommand{\em}{%
  \@nomath\em \if b\expandafter\@car\f@series\@nil
  \normalfont \else \bfseries\itshape \fi}
\makeatother

\newtheorem{theorem}[subsection]{Theorem}

\newtheorem{corollary}[subsection]{Corollary}
\newtheorem{lemma}[subsection]{Lemma}
\newtheorem{proposition}[subsection]{Proposition}

\newtheorem{question}{Question}
\theoremstyle{definition}

\newtheorem{example}[subsection]{Example}
\theoremstyle{remark}
\newtheorem{remark}[subsection]{Remark}

\numberwithin{equation}{section}
\numberwithin{figure}{section}

\newcommand{\B}[1]{{\mathbf #1}}
\newcommand{\C}[1]{{\mathcal #1}}

\newcommand{\OP}{\operatorname}

\definecolor{darkgreen}{rgb}{0.0, 0.5, 0.0}

\begin{document}
\title{Title}
\title[On quasimorphisms on homeomorphism groups]{On quasimorphisms and distortion in homeomorphism groups}
\author{Michael Brandenbursky}\thanks{M.B. was partially supported by the Israel Science Foundation grant 823/23}
\address{Department of Mathematics, Ben Gurion University, Israel}
\email{brandens@bgu.ac.il} 
\author{Jarek K\k{e}dra}
\address{Department of Mathematics, University of Aberdeen, UK and University of Szczecin, Poland}
\email{kedra@abdn.ac.uk}
\author{Michał Marcinkowski}
\address{Department of Mathematics, University of Wrocław, Poland}
\email{michal.marcinkowski@uni.wroc.pl}
\author{Egor Shelukhin}
\address{Department of Mathematics, University of Montreal, Canada}
\email{egorshel@gmail.com}

\begin{abstract}
\noindent Let $M$ be a smooth
compact oriented connected manifold, and ${\rm Homeo}_0(M,\mu)$ the group of homeomorphisms of $M$ supported away from $\partial M,$ which preserve a Borel probability measure $\mu$ induced by a volume form on $M$, and are isotopic to the identity. 
In this paper, we identify those Gambaudo-Ghys and Polterovich quasimorphisms
$\Psi\colon {\rm Diff}_0(M,\mu)\to\B R$ which extend $C^0$-continuously 
to ${\rm Homeo}_0(M,\mu)$ as quasimorphisms, and to ${\rm Homeo}_0(M)$ as group cochains whose differentials are semi-bounded cocycles.\\
\noindent We present several applications of this result which include unboundedness of 
certain bi-invariant metric on the commutator subgroup of ${\rm Homeo}_0(M,\mu)$, and conditions under which a homeomorphism in ${\rm Homeo}_0(M)$ is undistorted.
\end{abstract}

\maketitle

\section{Introduction}\label{S:intro}
Let $M$ be a smooth compact oriented connected manifold equipped with a  Borel probability 
measure $\mu$ defined by a volume form on $M$.
Let ${\rm Homeo}_0(M,\mu)$ denote the group of  
measure preserving homeomorphisms of $M$ which act by the  identity homeomorphism on a neighborhood of $\partial M$
and are isotopic to the identity. Let $P_n(M,z) = \pi_1({\rm C}_n(M),z)$ be the pure braid
group on $M$.  Here ${\rm C}_n(M) \subset M^n$ denotes the space of ordered
configurations of $n$ points in $M$, equipped with the product ${\rm Homeo}_0(M)$ action. 
If $\dim(M)=2$ then we let $n\in \B N$ be an
arbitrary number, and when $\dim(M)\geq 3,$ we assume $n=1$ throughout the paper. That is, for
higher dimensional manifolds we consider only their fundamental groups $P_1(M,z)=\pi_1(M,z)$.  

Let $\gamma\, \colon {\rm Homeo}_0(M)\times C_n(M)\to P_n(M)$ be a measurable cocycle defined by 
$$
\gamma(f,x) = \left[ \ell_{z,f(x)}*\{f_t(x)\}*\ell_{x,z} \right],
$$
where $\{f_t\}$ is an isotopy from the identity to $f$ and $\ell_{x,y}$ 
is a certain path in ${\rm C}_n(M)$ from $x$ to $y$ which is defined in
Section \ref{S:preliminaries}.  By convention, the concatenation of paths is
read from right to left.

The main object of study in the present paper is a family of maps
$\Psi\colon {\rm Homeo}_0(M)\to \B R$ defined by
\begin{equation}
\Psi(f) = \int_{{\rm C}_n(M)} \varphi(\gamma(f,x))dx,
\label{Eq:GG}
\end{equation}
where $\varphi\colon P_n(M)\to \B R$ is a homogeneous quasimorphism vanishing
on the centre of $P_n(M)$. On ${\rm C}_n(M)$ we consider the product measure $\mu^n$.  
The formula \eqref{Eq:GG} was used by
Gambaudo-Ghys to define non-trivial quasimorphisms on groups of area preserving
diffeomorphisms of compact surfaces \cite{zbMATH05019142} and by Polterovich for volume preserving 
diffeomorphisms of compact manifolds of higher dimension \cite{zbMATH05016104}. It was further exploited
by several authors in various configurations 
\cite{zbMATH06490154,zbMATH06160238,zbMATH06814502,zbMATH07159382,MR2253051}. In general, it is not
well defined for homeomorphisms and we identify quasimorphisms
$\varphi\colon P_n(M)\to \B R$ for which it is. Here is our main result.

\begin{theorem}\label{T:main}
Let $\varphi\colon P_n(M)\to \B R$ be a non-trivial homogeneous
quasimorphism vanishing on the centre of $P_n(M)$. If $M$ is a surface,
we assume that $M$ is closed of positive genus and $\varphi$ vanishes on the subgroup
$P_n(\Delta)\leq P_n(M)$ of braids supported in an interior of a full measure two-cell $\Delta$, see
Section \ref{S:preliminaries} for definition. Then the 
map $\Psi \colon {\rm Homeo}_0(M)\to \B R$ given by \eqref{Eq:GG} 
is well defined if and only if $\varphi\colon P_n(M)\to \B R$ is as above, and moreover:

\begin{enumerate}
\item The differential $\delta\Psi$ is a semi-bounded $2$-cocycle, that is
$$
D(f):=\sup_{g} |\delta\Psi(f,g)| = \sup_g |\Psi(g) - \Psi(fg) + \Psi(f)| < \infty.
$$
\item
The restriction of $\Psi$ to the subgroup of homeomorphisms preserving the measure
$\mu$ is a quasimorphism whose homogenization is $C^0$-continuous.
\item 
If $\varphi$ extends to an unbounded quasimorphism on the full braid group $B_n(M)$ 
then the homogeneous quasimorphism from the previous item is non-trivial.
\end{enumerate}
\end{theorem}

\begin{remark}\label{R:existence_of_psi}
Let $\Sigma$ be a closed surface of a positive genus and $n>1$. Recall that $B_n(\Sigma)/Z(B_n(\Sigma))$,
where $Z(B_n(\Sigma))$ is the center of $B_n(\Sigma)$, is a non-reducible subgroup  
of the $n$ punctured mapping class group $\operatorname{MCG}_n$ 
\cite[Corollary 7.13]{MR1195787} and obviously $P_n(\Delta)\leq \operatorname{MCG}_n$
is a reducible subgroup. It follows from Bestvina-Fujiwara construction \cite[Theorem 12]{zbMATH01928904} 
that the space of quasimorphisms on $B_n(\Sigma)$ 
which vanish on $P_n(\Delta)$ is infinite-dimensional. 
Hence there are infinitely many linearly independent quasimorphisms 
$\varphi\colon B_n(\Sigma)\to \B R$ which satisfy conditions of Theorem \ref{T:main}.

Let $q:\pi_1(\Sigma)\to\B R$ be a quasimorphism. It defines 
a quasimorphism $\varphi:P_n(\Sigma)\to\B R$ as follows. Let 
$$i_*:P_n(\Sigma)\to(\pi_1(\Sigma))^n$$
be the homomorphism induced by the inclusion $i: {\rm C}_n(\Sigma)\to(\Sigma)^n$. 
We define $\varphi=q\circ p_j\circ i_*$, where 
$p_j:(\pi_1(\Sigma))^n\to \pi_1(\Sigma)$ is the projection on the $j$-th factor.
It is a well-known fact that the kernel of the homomorphism $i_*$ equals to 
the normal closure $H_n$ of $P_n(\Delta)$ in the group $P_n(\Sigma)$ \cite[Theorem 1]{MR334182}, see also \cite{MR234447}.
Thus every such $\varphi$ vanishes on $H_n$. 

Let us consider the case when $\Sigma=T$ is a torus. 
In this case the group $H_n$ equals to the commutator subgroup $[P_n(T),P_n(T)]$, see \cite{MR234447}. 
Hence there are infinitely  many linearly independent quasimorphisms on $P_n(T)$, that vanish 
on $P_n(\Delta)$, which are obviously different from each such $\varphi$, since in this case 
$\varphi$ is a homomorphism. Let us discuss in more details the case $n=2$. 
The commutator subgroup $[P_2(T),P_2(T)]\cong \B F_2=\langle a,b \rangle$ and 
$P_2(\Delta)\cong\B Z$ is generated by the commutator $[a,b]$. Let 
$\operatorname{Q}(\B F_2;\B Z/2\times \B Z/2)$ be the subspace of quasimorphisms 
on $\B F_2$ invariant under the action generated by inverting generators.
Then every homogeneous quasimorphism $\phi\in\operatorname{Q}(\B F_2;\B Z/2\times \B Z/2)$ vanishes on $[a,b]$
and hence on $P_2(\Delta)$. Moreover, each such $\phi$ extends to $B_2(T)$, 
see \cite[Proposition 2.8]{zbMATH06814502}, hence by (3) in Theorem \ref{T:main} $\Psi$ is non-trivial. 
In particular, this gives an elementary construction of 
unbounded quasimorphisms on $\OP{Homeo}_0(T,\mu)$. 
Note that in this case, one needs to take $n>1$, since for $n=1$ we have $B_1(T) = Z(B_1(T))\cong\B Z^2$.

Let us discuss the remaining case when $\Sigma$ is a closed hyperbolic surface.
We would like to point out that there are infinitely many linearly independent quasimorphisms on $P_n(\Sigma)$, that vanish 
on $P_n(\Delta)$, which are different from each $\varphi$ described above.
Note that, there are infinitely many linearly independent quasimorphisms on 
$P_n(\Sigma)$ that vanish on $P_n(\Delta)$, but do not vanish on $H_n$. 
Indeed, by  \cite[Theorem 12]{zbMATH01928904} it is enough to show that the group $H_n$ 
is a non-reducible subgroup of $\operatorname{MCG}_n$. This can be shown as follows:
Let $C$ be a multicurve (a simplex in the curve complex). If it intersects non-trivially 
$\Delta\setminus z$, there is an element in $P_n(\Delta)$ which does not preserve $C$.
Otherwise, if $C$ lies outside of $\Delta$, then since $P_n(\Sigma)$ is non-reducible, there exists 
$\alpha\in P_n(\Sigma)$ so that $\alpha(C)\neq C$. Note that one can choose $\alpha$ such that $\alpha(C)$ intersects
$\Delta\setminus z$ non-trivially. Hence there exists $\beta\in P_n(\Delta)$ such that $\beta\alpha(C)\neq\alpha(C)$. It follows 
that $\alpha^{-1}\beta\alpha\in H_n$ and $\alpha^{-1}\beta\alpha(C)\neq C$. 
Thus $H_n$ is a non-reducible subgroup of $\operatorname{MCG}_n$.
\end{remark}

\begin{remark}
Semi-boundedness of $2$-cocycles was introduced by Gal-Kędra
\cite{zbMATH06065373} where it was used to prove undistortedness of symplectic
diffeomorphisms for certain symplectic manifolds. It was further generalized to
the concept of $p$-boundedness in \cite{arXiv:2208.03168}, which also
introduced $p$-bounded cohomology of groups. Functions with semi-bounded differentials were recently used in \cite{BS-equators} to study the $L^p$-geometry of the space of contractible loops on surfaces.
\end{remark}

\subsection*{\texorpdfstring{$C^0$}{C0}-continuity.}
Let $\mu$ be the measure associated with a volume form on $M$.
Given a homogeneous quasimorphism 
$\varphi \colon P_n(M)\to \B R$ vanishing
on the centre of $P_n(M)$, and in case $M$ is a surface vanishing on $P_n(\Delta)$, 
the formula \eqref{Eq:GG} well defines a
quasimorphism $\Psi\colon {\rm Diff}_0(M,{\mu})\to \B R$.

\begin{corollary}\label{C:c0-extension}
The homogenization $\overline{\Psi}\colon {\rm Diff}_0(M,{\rm \mu}) \to \B R$
has a unique $C^0$-continuous extension to 
${\rm Homeo}_0(M,\mu)$ which is a homogeneous quasimorphism.
\end{corollary}

\subsection*{Entropy.}
Let $\Sigma$ be a closed orientable surface of positive genus. For certain 
quasimorphisms $\varphi\colon B_n(\Sigma)\to \B R$ the restriction of
$\Psi$ to ${\rm Diff}_0(\Sigma,\mu)$ is bounded on the set ${\rm Ent}_0$ 
of diffeomorphisms of zero entropy \cite{zbMATH07159382}.
Moreover, these quasimorphisms come from Bestvina-Fujiwara construction and thus satisfy
conditions of Theorem \ref{T:main}, see \cite{zbMATH07159382}. 

It follows from Corollary \ref{C:c0-extension}
that the extension of the homogenization of $\Psi$,
$$\widetilde{\Psi}\colon {\rm Homeo}_0(\Sigma,\mu)\to \B R$$  
vanishes on the $C^0$-closure of ${\rm Ent}_0$ in ${\rm Homeo}_0(\Sigma,\mu)$
and on all its conjugates. 

Consider the set
$$
S = \bigcup_{f\in {\rm Homeo}_0(\Sigma,\mu)} f^{-1}\cdot \overline{\rm Ent}_0^{\rm Homeo}\cdot f,
$$
which is the normal closure of the $C^0$-closure in ${\rm Homeo}_0(\Sigma,\mu)$ of
the set of zero entropy measure preserving diffeomorphisms. Let $G_S \leq {\rm Homeo}_0(\Sigma, \mu)$ be the subgroup generated by $S.$ Note that $G_S$ is a large subgroup. For example, by simplicity of the kernel of the Flux homomorphism on  ${\rm Diff}_0(\Sigma,\mu)$ \cite{zbMATH01002916}, it is easily seen that ${\rm Diff}_0(\Sigma,\mu) \leq G_S$, see e.g. \cite{zbMATH07159382}. 
On the other hand, by \cite[Theorem 1.10]{CGHMSS}, 
the commutator group $[{\rm Homeo}_0(\Sigma, \mu), {\rm Homeo}_0(\Sigma, \mu)]$ is simple. This implies that 
$$[{\rm Homeo}_0(\Sigma, \mu), {\rm Homeo}_0(\Sigma, \mu)] \leq G_S.$$ 
Since $\widetilde{\Psi}$ does not vanish on ${\rm Diff}(\Sigma,\mu)$,
we have the following consequence of Corollary~\ref{C:c0-extension}.

\begin{corollary}\label{C:ent-metric}
The diameter of the bi-invariant metric on $G_S$ associated with the generating set $S$ is infinite.\qed
\end{corollary}

This corollary says that the set $S$ is {\it small}, in the sense that for any
positive integer $m\in \B N$ there are elements of ${\rm Homeo}_0(\Sigma,\mu)$ which
cannot be presented as product of up to $m$ elements from $S$. It follows from
\cite{zbMATH03676775} that $S$ is contained in a closed subset with empty interior, so it
is topologically small. The above corollary says that it is also algebraically
small. Since entropy is not $C^0$-continuous, the relation between $S$ and
the set of measure preserving homeomorphisms of zero entropy is unclear.

It is also unclear how large the group $G_S$ really is. Indeed, it contains all the elements which are currently known not to lie in $[{\rm Homeo}_0(\Sigma, \mu), {\rm Homeo}_0(\Sigma, \mu)]$. 
However, the abelianization ${\rm Homeo}_0(\Sigma, \mu)/[{\rm Homeo}_0(\Sigma, \mu), {\rm Homeo}_0(\Sigma, \mu)]$ is not yet completely determined and it is unclear what subgroup $G_S$ defines in it.

\begin{question}
Is it true that $G_S = {\rm Homeo}_0(\Sigma, \mu)$?   
\end{question}

\subsection*{Distortion.} 
Recall that an element $g$ of a finitely generated group
$\Gamma$ is undistorted if 
$$\lim_{k\to \infty}\frac{|g^k|}{k}>0,$$
where $|g|$ denotes the word norm with respect to any finite generating set of $\Gamma$.
An element $g$ of an arbitrary group $G$ is undistorted if it is undistorted in every
finitely generated subgroup $\Gamma\leq G$ containing it. Understanding distortion
in groups of homeomorphisms is motivated by the topological version of the Zimmer
conjecture. 

Existence of functions on a group $G$ whose differential
is semi-bounded can be used to prove undistortedness of elements of $G$.
The following observation almost immediately follows from
\cite[Proposition 4.2]{arXiv:2208.03168}.

\begin{theorem}\label{T:distortion}
Let $\psi \colon G\to \B R$ be a function such that $\delta\psi$ is semi-bounded.
If 
$$
\limsup_{k\to \infty}\frac{|\psi(f^k)|}{k}>0
$$ 
then $f$ is undistorted.
\end{theorem}

\begin{proof}
It follows from Lemma \ref{L:psi-norm} that $|f|_{\psi}=\sup_g|\psi(g)-\psi(fg)|$ 
is a pseudo-norm with respect to which
$\psi$ is Lipschitz which implies that 
\begin{equation}
\limsup_{k\to \infty} \frac{|f^k|_{\psi}}{k}>0.
\label{Eq:limsup}
\end{equation}
Let
$\Gamma\subseteq G$ be a finitely generated subgroup containing $f$
and let $C=\max\{|s|_{\psi}\ |\ s\in S\}$, 
where $S\subseteq \Gamma$ is a finite generating set. Then
$$
|f|_{\psi} = |s_1\dots s_m|_{\psi} \leq |s_1|_{\psi} + \dots + |s_m|_{\psi}
\leq C |f|_S
$$
which, together with \eqref{Eq:limsup} implies that $f$ is undistorted in $\Gamma$
and hence in $G$.
\end{proof}

\begin{corollary}\label{C:distortion}
Let $f\in {\rm Homeo}_0(M)$ and let $\Psi$ be defined by \eqref{Eq:GG}. If
$$
\limsup_{k\to \infty}\frac{|\Psi(f^k)|}{k}>0
$$
then $f$ is undistorted in ${\rm Homeo}_0(M)$.\qed
\end{corollary}

Verifying the hypothesis of the above corollary may not be easy in a concrete case.
However, we tackle this problem by redefining $\Psi$ in Theorem \ref{T:main} as follows. Let point
$z\in {\rm C}_n(M)$ be a base-point and let $\Psi_z\colon {\rm Homeo}_0(M)\to \B R$
be given by
$$
\Psi_z(f) = \varphi(\gamma(f,z)),
$$
where $\varphi\colon P_n(M)\to \B R$ is a homogeneous quasimorphism vanishing on $P_n(\Delta)$
and on the centre of $P_n(M)$. We show in Section \ref{S:distortion} that
the cocycle $\delta\Psi_z$ is semi-bounded (Proposition \ref{P:semi-bounded-z}).

\begin{corollary}\label{C:distortion-fix}
Let $f\in {\rm Homeo}_0(M)$. If $z\in M$ is a fixed point of $f$ such that
$\varphi(\gamma(f,z))>0$ then $f$ is undistorted in ${\rm Homeo}_0(M)$.
\end{corollary}

\begin{proof}
Since $z$ is a fixed point of $f$, we get that $\gamma(f^k,z) = \gamma(f,z)^k$.
It follows that 
$$\Psi_z(f^k)=\varphi(\gamma(f^k,z)) = k\Psi_z(f)>0$$ 
and the statement follows from Theorem \ref{T:distortion}.
\end{proof}

\begin{example}\label{E:hyperbolic}
Let $f$ be a homeomorphism of a closed oriented hyperbolic manifold. 
Suppose $z\in M$ is a fixed point of $f$ such that 
$$\gamma(f,z)\in \pi_1(M,z)$$
is non-trivial. Then $f$ is undistorted. Indeed, in this case there
exists a homogeneous quasimorphism non-vanishing on $\gamma(f,z)$ (see e.g. \cite{zbMATH06532523}) and
the statement follows from Corollary \ref{C:distortion}.
\hfill $\diamondsuit$
\end{example}

\subsection*{1-bounded cohomology.}
Since $\Psi$ is semi-bounded it is natural to ask whether it
represents a non-trivial class in $1$-bounded cohomology.
The next corollary states that this is indeed the case.

\begin{corollary}\label{C:h2b}
Let $M$ be a closed orientable surface of positive genus, or a higher dimensional manifold whose fundamental group modulo its centre admits infinitely many linearly independent homogeneous quasimorphisms. Then
$$
\dim H^2_{(1)}({\rm Homeo}_0(M),\B R) = \infty.
$$
\end{corollary}
\begin{proof}
The hypothesis implies that the construction \eqref{Eq:GG} yields infinitely many
linearly independent quasimorphisms on ${\rm Diff}_0(M,{\rm vol})$, according to
\cite{zbMATH06490154,zbMATH06814502}. In particular $H^2_b({\rm Diff}_0(M,{\rm vol});\B R)$ is infinite dimensional. Since the comparison map
$H^2_b(G;\B R)\to H^2_{(1)}(G;\B R)$ is injective for any group \cite{arXiv:2208.03168}, we obtain that the image of homomorphism 
$H^2_{(1)}({\rm Homeo}_0(M),\B R)\to H^2_{(1)}({\rm Diff}_0(M,{\rm vol});\B R)$ induced by the inclusion, contains an infinite dimensional part of the image of
the comparison map 
$H^2_b({\rm Diff}_0(M,{\rm vol};\B R)\to H^2_{(1)}({\rm Diff}_0(M,{\rm vol});\B R)$,
which proves the statement.
\end{proof}
\begin{remark}
The above statement for surfaces of positive genus also follows from the result of Bowden, Hensel and Webb
\cite[Theorem 1.4]{zbMATH07420183} who showed that $H^2_b({\rm Homeo}_0(\Sigma),\B R)$ is infinite dimensional.
Since the comparison homomorphism 
$H^*_b(G,\B R)\to H^*_{(1)}(G,\B R)$
is injective, their result implies the above. However, we conjecture that for surfaces our $1$-bounded
classes are not represented by bounded cocycles. In other words, they don't
come from quasimorphisms.
\end{remark}

\subsection*{Standard cohomology.}
The formula \eqref{Eq:GG} always defines a morphism between bounded
cochains 
$$C^*_b(P_n(\Sigma),\B R)\to C^*_b({\rm Homeo}_0(\Sigma,\mu),\B R),$$
which, according to \cite{zbMATH07503455}, descends to a homomorphism 
$$\C G\colon H^*_b(P_n(\Sigma),\B R)\to H^*_b({\rm Homeo}_0(\Sigma,\mu),\B R).$$ 
We thus have a commutative diagram
$$
\begin{tikzcd}
H^2_b(P_n(\Sigma);\B R)\ar[r,"\C G"] & H^2_b({\rm Homeo}_0(\Sigma,\mu),\B R)\ar[r,"c"] & H^2({\rm Homeo}_0(\Sigma,\mu),\B R)\\
Q(P_n(\Sigma))\ar[r, dashed]\ar[u,"\delta"] & Q({\rm Homeo}_0(\Sigma,\mu))\ar[ur,"0"] \ar[u,"\delta"] \\
\end{tikzcd}
$$
where $c$ is the comparison homomorphism,
the dashed arrow is not defined, and $Q$ denotes the space of quasimorphisms on a group in question. 
However, when homeomorphisms are replaced
by diffeomorphisms this arrow is well defined. In particular, for any quasimorphism
$\varphi$ the class 
$$c(\C G[\delta \varphi])=0$$
in $H^2({\rm Diff}_0(\Sigma,\mu))$. 
For homeomorphism we get classes $\C G[\delta\varphi]$ which are potentially
nontrivial in $H^2({\rm Homeo}_0(\Sigma,\mu),\B R)$.

\begin{question}
Does there exist $\varphi\colon P_n(\Sigma)\to \B R$ such that the
class $c(\C G[\delta\varphi])$ is non-trivial 
in the second cohomology group $H^2({\rm Homeo}_0(\Sigma,\mu),\B R)$?
\end{question}

\subsection*{Acknowledgments}
We would like to thank Juan Gonzalez-Meneses Lopez for helpful discussions.
All authors were partially supported by the OWRF fellowship in the MFO and the
RIG-fellowship of the ICMS. They wish to express their gratitude to the MFO and the ICMS for the support and excellent working conditions.

M.B. was partially supported by the Israeli Science Foundation grant 823/23.
E.S. was partially supported by an NSERC Discovery grant, the Courtois chair in fundamental research, an FRQNT teams grant, and a Sloan research fellowship.

\section{Preliminaries}\label{S:preliminaries}

\subsection*{Quasimorphisms.} Let $G$ be a group. A function $\varphi\colon G\to \B R$
is called a quasimorphism if
$$
D_{\varphi}:=\sup_{g,h}|\delta\varphi(g,h)| = \sup_{g,h}|\varphi(h) - \varphi(gh) + \varphi(g)| < \infty.
$$
In other words, $\delta\varphi$ is a bounded $2$-cocycle on $G$ with values in the
trivial module $\B R$. The constant $D_{\varphi}$ is called the defect of $\varphi$.
A function $\psi\colon G\to \B R$ is called homogeneous if 
$$
\psi(g^n)=n\psi(g)
$$
for every $g\in G$ and every $n\in \B Z$. A quasimorphism $\varphi$ can be homogenized 
by defining 
$$
\overline{\varphi}(g) = \lim_{n\to \infty}\frac{\varphi(g^n)}{n}.
$$
The above limit exists due to Fekete's Lemma and, moreover,
$\sup_g |\overline{\varphi}(g)-\varphi(g)|\leq D_{\varphi}$.
A homogeneous quasimorphism is constant on conjugacy classes. 
The differential $\delta\varphi$ of a quasimorphism represents a class of
the second bounded cohomology $H^2_b(G,\B R)$. If the homogenization of $\varphi$ 
is not a homomorphism then this class is non-trivial. 
Proofs of the above
facts are standard and can be found for example in \cite{zbMATH05577332}.\\

\subsection*{Semi-bounded cohomology.}
A $k$-cochain $c\colon G^k\to \B R$ is called $p$-bounded, where $p\leq k$, if for
fixed $g_1,\ldots,g_{k-p}\in G$ the function
$$
(g_{k-p+1},\ldots,g_k)\mapsto c(g_1,\ldots,g_k)
$$
is bounded. It follows that $p$-bounded cochains form subcomplexes and their
homology, denoted by $H^*_{(p)}(G,\B R)$, is called $p$-bounded (real)
cohomology of $G$. In particular, $k$-bounded cohomology is the bounded
cohomology $H^k_{(k)}(G,\B R)=H^k_b(G,\B R)$ and $0$-bounded $k$-th cohomology
is the ordinary one: $H^k_{(0)}(G,\B R) = H^k(G,\B R)$. Since $(p+1)$-bounded
cochains are $p$-bounded, there are comparison homomorphisms induced by
inclusion of complexes
$$
H^k_b(G,\B R)\to H^k_{(p+1)}(G,\B R)\to H^k_{(p)}(G,\B R)\to H^k(G,\B R).
$$
We refer to \cite{arXiv:2208.03168} for a general discussion and here we only discuss
$1$-bounded $2$-coboundaries. So, let $\psi\colon G\to \B R$ be a function
such that its coboundary is $1$-bounded:
$$
D(f):=\sup_{g}|\psi(g) - \psi(fg) + \psi(f)| <\infty.
$$

\begin{lemma}\label{L:psi-norm}
Let $\delta\psi$ be a $1$-bounded cocycle.
The formula
$$
|f|_{\psi} = \sup_g|\psi(g)-\psi(fg)|
$$
defines a pseudo-norm on $G$. Moreover, $\psi$ is $1$-Lipschitz with
respect to the pseudo-metric defined by $d_{\psi}(f,g) = |fg^{-1}|_{\psi}$.
\end{lemma}

\begin{proof}
Both the symmetry and triangle inequality is straightforward:
$$
|f^{-1}|_{\psi} = \sup_g|\psi(g)-\psi(f^{-1}g)| = \sup_g |\psi(f^{-1}g)-\psi(f\cdot f^{-1}g)|=|f|_{\psi}.
$$
\begin{align*}
|fg|_{\psi} &= \sup_h|\psi(h)-\psi(fgh)|\\
&\leq \sup_h|\psi(h)-\psi(gh)|+\sup_h|\psi(gh)-\psi(fgh)|
=|f|_{\psi}+|g|_{\psi}.
\end{align*}
As well as the Lipschitz property:
$$
|\psi(f)-\psi(g)| = |\psi(g)-\psi(fg^{-1}\cdot g)| \leq |fg^{-1}|_{\psi} = d_{\psi}(f,g),
$$
and the proof follows.
\end{proof}

\subsection*{The main construction.}
Let $d$ be an auxiliary Riemannian metric on $M$. It induces the supremum metric $d_0$ on ${\rm Homeo}_0(M)$ by 
$$ 
d_0(f,g):=\sup_{x\in M}d(f(x),g(x)).  
$$
Since $M$ is a compact connected smooth manifold it has a CW-decomposition with one top
dimensional cell. Let $\Delta\subseteq M$ be the interior of that cell
which is diffeomorphic to a ball in $\B R^n$ and it is of full measure in $M$.  
Choose a fundamental domain in
the universal cover of $M$ so that the lift of the inclusion $\Delta\subseteq
M$ is contained in it. 
Let $z=(z_1,\ldots,z_n)\in {\rm C}_n(M)$ be the base-point, such that
$z_i\in \Delta$ for each $i=1,\ldots,n$. 
Given two points $x,y\in {\rm C}_n(\Delta)$,
let $\ell_{y,x}$ be a path from $x$ to $y$ consisting of $n$
unit speed geodesics from $x_i$ to $y_i$ with respect to the Euclidean
metric on $\Delta$. This Euclidean metric is
used to define geodesics $\ell_{y,x}$ and has nothing to do with the
auxiliary metric chosen in the beginning of this section.

In case $M$ is a surface and $n>1$, 
then the geodesics between $x_i$ and $y_i$ may collide, 
and then $\ell_{x,y}$ is not defined. However, this happens only for a measure zero set, thus does not affect the integral from the definition of $\Psi$. For more details see \cite[Section 3.2]{GambaudoPecou}.
By convention, we read concatenation of paths from right to left. That
is, $\ell_{z,y}*\ell_{y,x}$ is a path from $x$ to $y$ and then to $z$.

Let $f\in {\rm Homeo}_0(M)$ and let $x\in {\rm C}_n(\Delta)$ be a point such that 
$f(x)\in {\rm C}_n(\Delta)$. Let a braid $\gamma(f,x)$ in $P_n(M,z)$ be 
represented by
\begin{equation}
\gamma(f,x) = \left[ \ell_{z,f(x)}*\{f_t(x)\}*\ell_{x,z} \right],
\label{Eq:gfz1}
\end{equation}
where $\{f_t\}$ is an isotopy from the identity to $f$. 

The above braids are defined for points $x\in M$ from a set of full measure.

If $P_n(M)$ has no centre then the braid $\gamma(f,x)$ does not depend on the choice
of an isotopy $\{f_t\}_{t=0}^1$ (see, e.g., \cite[Section 2.2 (6),(7)]{BK-frag}). 
Otherwise, it does depend and that is why we assume that
the homogeneous quasimorphism $\varphi$ from Theorem \ref{T:main} vanishes on the centre
and we need the following observation.\\
\begin{lemma}\label{L:normal}
Let $\varphi\colon G\to \B R$ be a homogeneous quasimorphism vanishing
on a normal subgroup $H\leq G$. Then
$$
\varphi(gh) = \varphi(g)
$$
for all $g\in G$ and $h\in H$.
\end{lemma}

\begin{proof}
Observe that for any two elements $g,h\in G$ we have 
$(gh)^k = g^kh^kc_1\dots c_k$,
where $c_i$ are conjugates of commutators of $h$ \cite[Section 2.2.4]{zbMATH05577332}. 
It follows that
\begin{align*}
|\varphi(gh) - \varphi(g)|
&=\frac{1}{k}\left|\varphi\left( (gh)^k \right)-\varphi(g^k)\right|\\
&=\frac{1}{k}\left|\varphi\left(g^kh^kc_1\dots c_k \right)-\varphi(g^k)\right|\\
&\leq\frac{1}{k}\left( |\varphi(g^k)+\varphi(h^kc_1\dots c_k) -\varphi(g^k) | +D_{\varphi}\right)
=\frac{D_{\varphi}}{k}.
\end{align*}
The last equality follows since $H$ is normal and hence $h^kc_1\dots c_k\in H$. 
Since $k\in \B N$ is arbitrary the computation proves the statement.
\end{proof}

Let $f,g\in {\rm Homeo}_0(M)$ and let $x\in {\rm C}_n(\Delta)$ be such that $g(x),fg(x)\in {\rm C}_n(\Delta)$.
Choosing the isotopy $\{(fg)_t\}$ to be the concatenation $\{f_t\circ g\}*\{g_t\}$ 
proves the following crucial fact.

\begin{lemma}\label{L:}
The function $\gamma\colon {\rm Homeo}_0(M)\times {\rm C}_n(\Delta)\to P_n(M)$ is a cocycle. That is,
$$
\gamma(fg,x) = \gamma(f,g(x))\gamma(g,x)
$$ 
for every $f,g\in {\rm Homeo}_0(M)$ and almost every $x\in {\rm C}_n(\Delta)$.
\qed
\end{lemma}

\section{Proofs}

\begin{proof}[Proof of Theorem \ref{T:main}]
If $n=1$, it is known that $\varphi(\gamma(f,x))$ is bounded as a function of $x$, thus the integral \eqref{Eq:GG} is well defined 
\cite{BK-frag}. Assume therefore that $M = \Sigma$ is a surface and $n$ is arbitrary.
Let $P_n(\Delta) = \pi_1({\rm C}_n(\Delta))$. 
The inclusion $\Delta\subseteq \Sigma$ induces an inclusion $P_n(\Delta)\leq P_n(\Sigma)$, see \cite[Theorem 1]{MR334182}.
Let $f\in {\rm Homeo}_0(\Sigma)$ and let 
$\varphi\colon P_n(\Sigma)\to \B R$ be a quasimorphism which
vanishes on $P_n(\Delta)$ and on the centre of $P_n(\Sigma)$. 
Let  ${\rm sys}(\Sigma,d)$ denote the systole of $\Sigma$.
Suppose that $f$ is $C^0$-small, in the sense, that
$d_0(f,{\rm Id}) < \epsilon$
for $\epsilon$ chosen such that there exists an isotopy $f_t$, connecting the identity to $f$, such that 
$$d_0(f_t,{\rm Id}) < \frac{1}{2n}{\rm sys}(\Sigma,d)$$ 
for all $t$. Existence of such an $\epsilon$ follows from local 
contractibility of ${\rm Homeo}_0(\Sigma)$ \cite[Corollary 1.1]{zbMATH03341019}. 
For every $x$, we shall construct a decomposition $\gamma(f,x) = \alpha\cdot\beta$,
where $\beta \in P_n(\Delta)$ and $\alpha\in P_n(\Sigma)$ is a braid from a certain fixed finite set. 

\begin{figure}[h]
\centering
\begin{tikzpicture}[line width=3pt, scale=0.25]
\draw[blue] (15,5) -- (5,15);
\draw[brown] (15,5) -- (9,17);
\draw[green] (5,15) ..controls (4,18) and (8,20).. (9,17);
\filldraw[red]
(15,5) circle (7pt)
(5,15) circle (5pt)
(9,17) circle (5pt);
\draw (7,15) node {$x$};
\draw (12,17) node {$f(x)$};
\draw (17,5) node {$z$};
\draw (2,2) node {$\Delta$};
\draw(0,0) rectangle (20,20);
\end{tikzpicture}
\caption{Isotopy inside $\Delta$.}
\label{F:in-delta}
\end{figure}
Let $x=(x_1,\ldots,x_n)$. Assume first, that the evaluation
$\{f_t(x_i)\}$ of an isotopy $\{f_t\}$ is contained in $\Delta$ for all
$i=1,\ldots,n$. Then $\gamma(f,x)\in P_n(\Delta)$. 
In this case $\beta = \gamma(f,x)$ and $\alpha = e$. A simplified picture
is presented in Figure \ref{F:in-delta}. The points $z,x$ and $f(x)$
should be seen as $n$-tuples, the path connecting $z$ and $x$ (blue) denotes the union of geodesics connecting $z_i$ and $x_i$, the path connecting $f(x)$ and $z$ (brown) denotes the union of geodesics
connecting $f(x_i)$ and $z_i$, and the path between $x$ and $f(x)$ (green) denotes the union of images of the isotopy
$\{f_t\}$ evaluated at $x_i$'s which is possibly braided.

In general, since $f$ is $C^0$-small,
we can choose an isotopy $f_t$, such that for every $i$ the path $\{f_t(x_i)\}$ is contained in a ball $B_i$ centered in $x_i$, and of radius smaller than $\frac{1}{2n}{\rm sys}(\Sigma,d)$. 
Recall that 
\begin{equation}
\gamma(f,x) = \left[ \textcolor{brown}{\ell_{z,f(x)}}*\textcolor{darkgreen}{\{f_t(x)\}}*\textcolor{blue}{\ell_{x,z}} \right],
\end{equation}
 see Figure \ref{F:back-and-forth}. Our goal is to move $\{f_t(x)\}$ (the green part) inside ${\rm C}_n(\Delta)$ and modify $\ell_{z,f(x)}$ (the brown part) in a controlled way. To this end, for each $i$ we choose a point $y_i \in \Delta \cap B_i$ such that the unit speed geodesic segment $\sigma_{y_i,f(x_i)}$ connecting $f(x_i)$ to $y_i$ lies entirely in $B_i$ (in particular, we may take $y_i = x_i$). Note that some of these geodesic segments might intersect in the configuration space, but it happens only for measure zero set of points $x \in C_n(\Delta)$. Let $y = \{y_1,\ldots,y_n\}$, and let $\ell_{y,f(x)}$ denote the path in ${\rm C}_n(\Sigma)$ consisting of segments $\sigma_{y_i,f(x_i)}$. We have

\begin{equation}
\gamma(f,x) = \left[ \textcolor{brown}{\ell_{z,f(x)}}*\textcolor{brown}{\ell_{f(x),y}}*\textcolor{darkgreen}{\ell_{y,f(x)}}*\textcolor{darkgreen}{\{f_t(x)\}}*\textcolor{blue}{\ell_{x,z}} \right],
\end{equation}
where $\ell_{f(x),y}$ denotes the time-reverse of $\ell_{y,f(x)}$.

\begin{figure}[h]
\centering
\begin{tikzpicture}[line width=3pt, scale=0.25]
\draw[blue] (15,5) -- (5,15);
\draw[brown] (15,25) -- (9,22);
\draw[green] (5,15) ..controls (4,24) and (7,23).. (9,22);
\filldraw[red]
(15,5) circle (7pt)
(15,25) circle (7pt)
(5,15) circle (5pt)
(9,22) circle (5pt);
\draw (7,15) node {$x$};
\draw (12,22) node {$f(x)$};
\draw (17,5) node {$z$};
\draw (2,2) node {$\Delta$};
\draw(0,0) rectangle (20,20);
\begin{scope}[shift={(30,0)}]
\draw[blue] (15,5) -- (5,15);
\draw[brown, rounded corners] (15,25) -- (9,22) -- (9,17);
\draw[green, rounded corners] (5,15) ..controls (4,24) and (7,23).. (8.5,22) -- (8.5,17);
\filldraw[red]
(15,5) circle (7pt)
(15,25) circle (7pt)
(5,15) circle (5pt)
(8.75,17) circle (5pt);
\draw (7,15) node {$x$};
\draw (17,5) node {$z$};
\draw (9.5,15.5) node {$y$};
\draw (2,2) node {$\Delta$};
\draw(0,0) rectangle (20,20);
\end{scope}
\end{tikzpicture}
\caption{Back-and-forth.}
\label{F:back-and-forth}
\end{figure}

The path $\ell_{y,f(x)}*\{f_t(x)\}$, in green in Figure \ref{F:back-and-forth},
is contained in ${\rm C}_n(\bigcup B_i)$. Moreover, since each $B_i$ has radius less than $\frac{1}{2n}{\rm sys}(\Sigma,d)$, every connected component of $\bigcup B_i$ has diameter less than ${\rm sys}(\Sigma)$ and is therefore contractible in $\Sigma$. Consequently, we can homotope $\ell_{y,f(x)}*\{f_t(x)\}$ inside ${\rm C}_n(\Delta)$, keeping the endpoints fixed:
\begin{equation}
\gamma(f,x) = \left[ \textcolor{brown}{\ell_{z,f(x)}}*\textcolor{brown}{\ell_{f(x),y}}*\textcolor{darkgreen}{s}*\textcolor{blue}{\ell_{x,z}} \right],
\end{equation}
where $s$ denotes $\ell_{y,f(x)}*\{f_t(x)\}$ after being pushed into ${\rm C}_n(\Delta)$, 
see Figure \ref{F:back-and-forth2}. The final step is to connect the endpoint of $s$ to $z$, and return back. We have:

\begin{equation}
\gamma(f,x) = \left[ \textcolor{brown}{\ell_{z,f(x)}}*\textcolor{brown}{\ell_{f(x),y}}*\textcolor{brown}{\ell_{y,z}}*\textcolor{darkgreen}{\ell_{z,y}}*\textcolor{darkgreen}{s}*\textcolor{blue}{\ell_{x,z}} \right].
\end{equation}

\begin{figure}[h]
\centering
\begin{tikzpicture}[line width=3pt, scale=0.25]
\draw[blue] (15,5) -- (5,15);
\draw[brown, rounded corners] (15,25) -- (9,22) -- (9,17);
\draw[green, rounded corners] (5,15) ..controls (4,20) and (6,19).. (8.5,17);
\filldraw[red]
(15,5) circle (7pt)
(15,25) circle (7pt)
(5,15) circle (5pt)
(8.75,17) circle (5pt);
\draw (7,15) node {$x$};
\draw (17,5) node {$z$};
\draw (9.5,15.5) node {$y$};
\draw (2,2) node {$\Delta$};
\draw(0,0) rectangle (20,20);
\begin{scope}[shift={(30,0)}]
\draw[blue] (15,5) -- (5,15);
\draw[brown, rounded corners] (15,25) -- (9,22) -- (9,17) -- (15.5,5);
\draw[green, rounded corners] (5,15) ..controls (4,20) and (6,19).. (8.5,17) -- (15,5);
\filldraw[red]
(15.25,5) circle (7pt)
(15,25) circle (7pt)
(5,15) circle (5pt);
\draw (7,15) node {$x$};
\draw (17,5) node {$z$};
\draw (13,23) node {$\alpha$};
\draw (9,9) node {$\beta$};
\draw (2,2) node {$\Delta$};
\draw(0,0) rectangle (20,20);
\end{scope}
\end{tikzpicture}
\caption{Push into $\Delta$ and back-and-forth again.}
\label{F:back-and-forth2}
\end{figure}

We define $\beta$ to be the paths 
$\ell_{z,y}*s*\ell_{x,z}$
obtained by the concatenation of the blue and green paths, and the brown path
$\ell_{z,f(x)}*\ell_{f(x),y}*\ell_{y,z}$ comprise the braid $\alpha$. By Lemma \ref{L:fin_many_values} proven below, 
$\varphi$ is bounded on all such braids $\alpha$. 

Thus, since the quasimorphism $\varphi$ vanishes on $P_n(\Delta)$, we obtain
\begin{equation}
|\varphi(\gamma(f,x))| = |\varphi(\beta\alpha)|
\leq |\varphi(\alpha)| + D_{\varphi},
\label{Eq:bounded-fi-gamma} 
\end{equation}
and hence the function $x\mapsto |\varphi(\gamma(f,x))|$ is bounded,
and so the integral
$$\Psi(f) = \int_{{\rm C}_n(\Sigma)}\varphi(\gamma(f,x))dx$$
is well defined, i.e. $\varphi(\gamma(f,x))$ is a $L^1$-function, 
for a $C^0$-small $f$.

Let $f\in {\rm Homeo}_0(\Sigma)$ be an arbitrary homeomorphism and let
$\{f_t\}$ be an isotopy from the identity to $f$, where $t\in [0,1]$.
Let $0=t_0<t_1<t_2 <\dots < t_m=1$ be a partition of the interval $[0,1]$
such that 
$$d_0(f_{t_i}^{-1}f_{t_{i+1}},{\rm Id}) < \epsilon,$$ 
where $\epsilon$ is chosen as in the beginning of the proof, i.e., $f_{t_i}^{-1}f_{t_{i+1}}$ are $C^0$-small for all $i$. Of course,
$$
f = f_{t_0}\circ f_{t_0}^{-1}f_{t_1}\circ f^{-1}_{t_1}f_{t_2}\circ \dots \circ f_{t_{m-1}}^{-1}f_{t_m}
=g_1\dots g_m,
$$
where $g_i = f_{t_{i-1}}^{-1}f_{t_i}$.
That is, $f$ is a product of $C^0$-small homeomorphisms. We get 
\begin{align*}
\varphi(\gamma(f,x))
&=\varphi(\gamma(g_1\dots g_m,x))\\
&=\varphi(\gamma(g_1,g_2\dots g_mx)\gamma(g_2,g_3\dots g_mx)\dots \gamma(g_m,x))\\
&=\varphi(\gamma(g_1,g_2\dots g_mx))+ \varphi(\gamma(g_2,g_3\dots g_mx))+ \dots +\varphi(\gamma(g_m,x))
+ d_1+\dots +d_{m-1},
\end{align*}
where $|d_i|\leq D_{\varphi}$. 
This implies that $\varphi(\gamma(f,x))$ attains finitely many values for
almost every $x\in \Delta$. 
As before, it follows that $\Psi(f)$ is well defined.  

Conversely, assume that the integral \eqref{Eq:GG} is well defined for
all $f\in {\rm Homeo}_0(\Sigma)$. 
Suppose that $\varphi \colon P_n(\Sigma)\to \B R$
does not vanish on a braid $\beta\in P_n(\Delta)$. Let $f_1\in {\rm Homeo}_0(\Sigma)$
be a homeomorphism supported in a ball $B_1\subset \Delta$ of area 
$$0<a<\frac{1}{10}\mu(\Delta)$$
such that 
$$\gamma(f,x)=\beta^{10}$$
for $x\in B_1$ in the set of area $b>0$.
Let $f_2$ be a homeomorphism supported in a ball $B_2 \subset \Delta$ disjoint from $B_1$ and of area
$a/2$ such that 
$$\gamma(f,x)=\beta^{10^2}$$
for $x\in B_2$ in the set of area $b/2$.
By iterating this construction and taking the limiting homeomorphism 
$f=f_1\circ f_2\circ f_3\dots$ we get that the integral \eqref{Eq:GG} diverges which
contradicts our hypothesis. This shows the main statement of Theorem \ref{T:main}.
We now prove the additional statements. 

\noindent
{\bf (1)} Let $A(f) = \max\{\varphi(\gamma(f,x))\ |\ x\in {\rm C}_n(M)\}$. It is
a well defined number because $\varphi(\gamma(f,x))$ is bounded as explained in the first part of the proof.
The following computations shows that $\delta\Psi$ is semi-bounded.

\begin{align*}
|\delta\Psi(f,g)|
&=|\Psi(g)-\Psi(fg)+\Psi(f)|\\
&\leq \int_{{\rm C}_n(M)} |\varphi(\gamma(g,x)) -\varphi(\gamma(fg,x)) + \varphi(\gamma(f,x))|dx\\
&\leq \int_{{\rm C}_n(M)} \big |\varphi(\gamma(g,x)) -\varphi(\gamma(f,gx)\gamma(g,x)) + \varphi(\gamma(f,x))\big |dx\\
&\leq \int_{{\rm C}_n(M)} \left(\big |\varphi(\gamma(g,x)) -\varphi(\gamma(f,gx))-\varphi(\gamma(g,x)) + \varphi(\gamma(f,x))\big |+D_{\varphi}\right)dx\\
&\leq \int_{{\rm C}_n(M)} \left(\big | \varphi(\gamma(f,gx))- \varphi(\gamma(f,x))\big |+D_{\varphi}\right)dx\\
&\leq \int_{{\rm C}_n(M)} (2A(f)+D_{\varphi})dx <\infty.
\end{align*}

\noindent
{\bf (2)} The quasimorphism property for the restriction of $\Psi$ to 
area preserving homeomorphism is a standard computation as in \cite{zbMATH05980986}.
The $C^0$-continuity of the homogenization $\overline{\Psi}$ follows
from a theorem of Shtern \cite[Theorem 1]{zbMATH01709147}, which states that a homogeneous
quasimorphism on a topological group is continuous if and only if
it is bounded on a neighborhood of the identity. Indeed, if
$f$ is $C^0$-small then
$$
|\varphi(\gamma(f,x))| = |\varphi(\beta\alpha)| \leq |\varphi(\alpha)|+D_{\varphi},
$$
which is bounded independently of $f$. Hence $\Psi(f)$ is bounded on
a neighborhood of the identity. Since 
$$\sup_{f}|\overline{\Psi}(f)-\Psi(f)|\leq D_{\Psi},$$
we get that $\overline{\Psi}$ is also bounded on a neighborhood of the
identity and its continuity follows from Shtern's theorem.

\noindent
{\bf (3)} Assume $n=1$ and let $\gamma$ be a braid such that 
$\overline{\varphi}(\gamma) \neq 0$. Let $f$ be a point pushing homeomorphism $f$ 
along the loop $\gamma$ such that $\gamma(f,x) \in \{\gamma,e\}$, where $e$ is the trivial 
loop, for all $x$ away from a set of arbitrary small measure. 
This gives us $\overline{\Psi}(f) \neq 0$ (see, e.g., \cite[Theorem 2.5]{zbMATH07159382}). 
If $n>1$ and $\varphi$ extends to the full braid group $B_n(\Sigma)$ then the unboundedness
of $\Psi$ is proven in \cite{zbMATH06490154} extending an argument of
Ishida \cite{zbMATH06596006}.
\end{proof}

We finish this section with the technical lemma used in the proof of Theorem \ref{T:main}.

\begin{lemma}\label{L:fin_many_values}
Let $\Sigma$ be a closed surface of positive genus 
and let $\varphi$ be a quasimorphism vanishing on the centre of $P_n(\Sigma)$. 
Fix an integer $N\geq 1$. Then $\varphi$ is bounded on the set of braids $\alpha\in P_n(\Sigma)$ whose
every strand is a concatenation of at most $N$ geodesic segments, each one of
length at most $\operatorname{diam}(\Sigma)$.
\end{lemma}

\begin{proof}
Let \(Z=\{z_1,\ldots,z_n\}\) and put \(S=\Sigma\setminus Z\). Consider the point-pushing homomorphism
\[
h \colon P_n(\Sigma)\longrightarrow \operatorname{PMod}(S),
\]
where $\operatorname{PMod}(S)$ is the pure mapping class group of $S$. 
By Birman exact sequence (see, e.g., \cite{MR243519} and \cite{MR2850125}), 
the kernel of $h$ is the centre of $P_n(\Sigma)$. 
Since $\varphi$ vanishes on the centre, it is enough to 
show that the set of point-pushing homeomorphisms $h(\alpha)$, 
where $\alpha$ is as in the lemma, is finite. 

Choose a finite system $\mathcal A=\{a_1,\ldots,a_m\}$ of essential simple
closed curves and properly embedded arcs in $S$, with endpoints at the
punctures allowed, such that every component of
$S\setminus \mathcal A$ is a disc and $\mathcal A$ satisfies the assumptions of \cite[Proposition 2.8]{MR2850125}.
It means that by definition \(\mathcal A\) fills \(S\).
By the Alexander method, an element $h$ of $\operatorname{PMod}(S)$ is determined, 
up to finite ambiguity depending only on $\mathcal{A}$,
by the isotopy classes of the curves and arcs $h(a_1),\ldots,h(a_m)$.
We show that, for the point-pushing classes \(h_\alpha = h(\alpha)\) arising from
the braids \(\alpha\), each \(h_\alpha(a_j)\) belongs to a finite set of
isotopy classes.
Fix \(a_j\in \mathcal A\). The image \(h_\alpha(a_j)\) is obtained from
\(a_j\) by pushing the punctures along the strands of \(\alpha\). An example of 
$h_\alpha(a_j)$ is shown in Figure \ref{F:back-and-forth3}. Away from a
small regular neighborhood of the strands the arc or curve is unchanged.
Inside this neighborhood, every change is caused by one of the finitely many
local events that corresponds to crossings of strands with \(a_j\) 
and crossings between strands. Since the strands of $\alpha$ are 
sums of at most $N$ short geodesic segments, there exists a bound  
on the number of such crossings, and it does not depend on $\alpha$. 
Thus there is a constant \(C\), such that $i\bigl(h_\alpha(a_j),\mathcal A\bigr)\leq C$ 
for every \(j\) and for every braid \(\alpha\) as in the lemma.
\begin{center}
\begin{figure}
\begin{tikzpicture}[
    x=0.38cm,
    y=0.38cm,
    line cap=round,
    line join=round,
    bluecurve/.style={draw=blue!80!black, line width=1.1pt},
    redcurve/.style={draw=red!85!black, line width=1.1pt},
    punct/.style={circle, fill=red!85!black, inner sep=1.5pt}
]

\begin{scope}[shift={(0,0)}]
    \draw[bluecurve] (0,0) rectangle (12,17);

    \draw[bluecurve] (0,14) -- (12,14);
    \node[blue!80!black] at (2.2,15.2) {$a_j$};

    \draw[redcurve] (6,17) -- (6,4.2);
    \draw[redcurve] (12,10) -- (2,10);

    \node[punct] at (6,4.2) {};
    \node[punct] at (2,10) {};

    \node at (6,3.5) {$z_1$};
    \node at (1.7,9) {$z_2$};
\end{scope}

\begin{scope}[shift={(13,0)}]
    \draw[bluecurve] (0,0) rectangle (12,17);

    \draw[bluecurve] (0,14) -- (4,14) -- (6,1) -- (8,14) -- (12,14);
    \node[blue!80!black] at (2.2,15.2) {$a_j$};

    \draw[redcurve] (6,17) -- (6,4.2);
    \draw[redcurve] (12,10) -- (2,10);

    \node[punct] at (6,4.2) {};
    \node[punct] at (2,10) {};

    \node at (6,3.5) {$z_1$};
    \node at (1.7,9) {$z_2$};
\end{scope}

\begin{scope}[shift={(26,0)}]
    \draw[bluecurve] (0,0) rectangle (12,17);

    \draw[bluecurve] (0,14) -- (4,14) -- (4.25,12.375) -- (0.1,10) -- (5,7.625) -- (6,1) -- (7,7.625) -- (1,10) --(7.75,12.375) -- (8,14) -- (12,14);
    \node[blue!80!black] at (2.2,15.2) {$a_j$};

    \draw[redcurve] (6,17) -- (6,4.2);
    \draw[redcurve] (12,10) -- (2,10);

    \node[punct] at (6,4.2) {};
    \node[punct] at (2,10) {};

    \node at (6,3.5) {$z_1$};
    \node at (1.7,8.5) {$z_2$};
\end{scope}
\end{tikzpicture}
\caption{An example of $h_\alpha(a_j)$. The two strands have one crossing and the strand that ends at $z_1$ crosses first.}
\label{F:back-and-forth3}
\end{figure}
\end{center}

Finally, since  \(\mathcal A\) fills \(S\), 
there are only finitely many isotopy classes of essential simple closed curves and essential arcs whose
geometric intersection number with \(\mathcal A\) is at most \(C\). To see
this, cut \(S\) along \(\mathcal A\). A curve or arc in minimal position with
\(\mathcal A\) is then encoded by a finite word recording the sequence of
sides of the complementary discs which it crosses; the length of this word is
bounded by \(C\), and the alphabet is finite. Thus each
\(h_\alpha(a_j)\) has only finitely many possible isotopy classes.
It follows that the tuple 
$$\bigl(h_\alpha(a_1),\ldots,h_\alpha(a_m)\bigr)$$
has only finitely many possible values. By the Alexander method, only finitely
many mapping classes \(h_\alpha\in \operatorname{PMod}(S)\) can arise. 
Since \(\varphi\) vanishes on the centre, the set
of possible values of \(\varphi(\alpha)\) is finite.
\end{proof}

\begin{remark}
Note that for a closed hyperbolic surface $\Sigma$ there are finitely many braids in $P_n(\Sigma)$ whose
every strand is a concatenation of at most $N$ geodesic segments, each one of
length at most $\operatorname{diam}(\Sigma)$. Indeed, if $\Sigma$ is a closed hyperbolic surface, 
then $P_n(\Sigma)$ has no centre, the push-point homomorphism is injective, 
and the claim follows from Lemma \ref{L:fin_many_values}. 
\end{remark}

\section{Distortion}\label{S:distortion}

Let $z\in {\rm C}_n(M)$ be the base-point and let 
$\Psi_z\colon {\rm Homeo}_0(M)\to \B R$ 
be the composition 
$$\Psi_z(f) = \varphi(\gamma(f,z)),$$
where $\varphi\colon P_n(M)\to \B R$ is a homogeneous quasimorphism vanishing on the
centre of $P_n(M)$ (recall that if the dimension of $M$ is greater than $2$ we assume $n=1$). 
If $M$ is a surface we additionally assume that $\varphi$ vanishes on $P_n(\Delta)$. Recall that under this
assumptions the function $x\mapsto|\varphi(\gamma(f,x))|$ is bounded, according to \eqref{Eq:bounded-fi-gamma}.

\begin{proposition}\label{P:semi-bounded-z}
The differential $\delta\Psi_z$ is a semi-bounded cocycle.
\end{proposition}

\begin{proof}
Let $B_f = \sup_{x}|\varphi(\gamma(f,x))|$.
For $f,g\in {\rm Homeo}_0(M)$ we have
\begin{align*}
|\delta\Psi_z(f,g)|
&=|\varphi(\gamma(g,z)) - \varphi(\gamma(fg,z)) + \varphi(\gamma(f,z))|\\
&\leq |\varphi(\gamma(g,z)) - \varphi(\gamma(f,g(z))) - \varphi(\gamma(g,z))+\varphi(\gamma(f,z))| + D_{\varphi}\\
&=|\varphi(\gamma(f,g(z))) - \varphi(\gamma(f,z)) | + D_{\varphi}\\
&=
\begin{cases}
|\varphi(\gamma(f,g(z))-\varphi(\gamma(f,z))|+D_{\varphi} &\text{ if } g(z)\neq z\\
D_{\varphi} &\text{ if } g(z)=z
\end{cases}\\
&\leq 2B_f+D_{\varphi}.
\end{align*}
This shows that $\delta\Psi$ is $1$-bounded.
\end{proof}

As explained in the introduction, if 
$\limsup_{k\to \infty}\frac{|\Psi_z(f^k)|}{k}>0$
then $f$ is undistorted in ${\rm Homeo}_0(M)$, 
according to Theorem \ref{T:distortion}.
The following proposition is a direct application
and its proof is straightforward.

\begin{proposition}\label{P:non-fix}
Let $f\in {\rm Homeo}_0(M)$ and $x\in M$ be such that
\begin{equation}
\gamma(f^{n_k},x) = \gamma(f,x)^{n_k}
\label{Eq:non-fix}
\end{equation}
for an increasing sequence $\{n_k\}_{k\in \B N}$ of positive integers.
If $\varphi$ is a homogeneous quasimorphism on $P_n(M)$ satisfying the
condition of Theorem \ref{T:main} and
$\varphi(\gamma(f,x))>0$ then $f$ is undistorted in ${\rm Homeo}_0(M)$.
\qed
\end{proposition}

\begin{example}\label{E:fix-point}
The hypothesis \eqref{Eq:non-fix} of the above proposition is immediate if
$x\in M$ is a fixed point of $f$. If not then there exists a homeomorphism
$s\colon M\to M$ and its isotopy $\{s_t\}$ to the identity such that
$t\mapsto s_t(f(x))$ is equal to $\ell_{x,f(x)}$. It follows that
$x$ is a fixed point of $sf$ and 
$$\gamma(sf,x) = \gamma(f,x).$$
Moreover, 
$$\gamma\left( (sf)^k,x \right) = \gamma(sf,x)^k = \gamma(f,x)^k.$$
If $\varphi(\gamma(f,x))>0$ then $sf$ is undistorted, according to
Proposition \ref{P:non-fix}.
\hfill $\diamondsuit$
\end{example}

The above example shows that the dynamics of distorted homeomorphisms is often
very restricted, see also \cite{zbMATH06252351}. Let us discuss a family of examples.  We say that a group $G$
is quasi-residually real 
if for every non-trivial element $g\in G$ there exists a
homogeneous quasimorphism $\varphi\colon G\to \B R$ such that $\varphi(g)\neq
0$. Examples of such groups include right-angled Artin groups, pure braid
groups and many hyperbolic groups \cite{zbMATH06532523}.

\begin{corollary}\label{C:qrr}
Let $M$ be a closed manifold with quasi-residually real 
fundamental group $\pi_1(M)$ with trivial centre. Let $f\in {\rm Homeo}_0(M)$
be a distorted element.
Then every fixed point $x\in M$ of $f$ is contractible.
That is, the loop $\{f_t(x)\}$ is contractible for every choice
of isotopy from the identity to $f$.
In particular, given a fixed point $x$ of $f$, there is an isotopy from the identity to $f$
fixing $x$.
\end{corollary}
\begin{proof}
Let ${\rm Homeo}_0(M,x)=\{g\in {\rm Homeo}_0(M)\ |\ g(x)=x\}$.
Consider the evaluation fibration 
$$
{\rm Homeo}_0(M,x) \to {\rm Homeo}_0(M)\stackrel{{\rm ev}}\longrightarrow \pi_1(M),
$$
where ${\rm ev}(g) = g(x)$. 
It follows from a theorem of Gottlieb \cite[Theorem I.4]{zbMATH03239188} that the
image of the homomorphism induced by the evaluation on the fundamental group is
contained in the center of $\pi_1(M,x)$. Since this center is
trivial and the group ${\rm Homeo}_0(M)$ is connected, the connecting homomorphism
$$
\partial\colon \pi_1(M,x)\to \pi_0({\rm Homeo}_0(M,x))
$$ 
is an isomorphism
of groups. Since $f$ is distorted and $\pi_1(M)$ is quasi-residually real,
it follows from Proposition \ref{P:non-fix} that $\varphi(\gamma(f,x))=0$ for
every homogeneous quasimorphism on $\pi_1(M)$. Since the latter group is quasi-residually
real, it follows that
the element $\gamma(f,x)$ is trivial. This implies that $f$ belongs
to the connected component of the identity of ${\rm Homeo}_0(M,x)$ and hence
there is an isotopy from the identity to $f$ fixing $x$.
\end{proof}


\begin{remark}
If we drop the hypothesis about trivial centre of $\pi_1(M)$ then the element
$\gamma(f,x)$ is not well defined and it depends on the choice of isotopy
$\{f_t\}$ from the identity to $f$. However, the conclusion in this case
is that the loop $\{f_t(x)\}$ represents a central element
of $\pi_1(M)$.
\end{remark}

A homeomorphism $f$ of a manifold $M$ is called {\it recurrent}
if there exists an increasing sequence $\{n_k\}_{k=1}^\infty$ of natural numbers
such that 
$$\lim_{k\to \infty} d_0(f^{n_k},{\rm Id}) = 0.$$
In other words, arbitrary large powers of $f$ are arbitrarily close to the
identity. For example, if $f$ is periodic then it is obviously recurrent.
Another example is an irrational rotation of the circle.  The only recurrent
diffeomorphisms of surfaces of higher genus are periodic, according to \cite{zbMATH01202977}.

\begin{proposition}\label{P:recurrent}
Let $f\in {\rm Homeo}_0(M)$ be a recurrent homeomorphism of
a closed manifold. Suppose that $\pi_1(M,z)$ has trivial centre.
If $z \in M$ is a fixed point of $f$, then
$\gamma(f,z)\in \pi_1(M,z)$ is torsion.
\end{proposition}

\begin{proof}
Since $z$ is a fixed point of $f$, we have $\gamma(f^k,z)=\gamma(f,z)^k$.
On the other hand, since $f$ is recurrent, there exists 
$k_0\in \B N$ such that $f^{k_0}$ is $C^{0}$-close to
the identity. Hence $\gamma(f^{k_0},z)=1$. 
\end{proof}

Militon proved that every recurrent {\it diffeomorphism} of a closed
manifold $M$ is distorted \cite{zbMATH06161942}. The recurrence of a
diffeomorphism is meant here with respect to the $C^{\infty}$-topology.

\bibliographystyle{plain}
\bibliography{bibliography}

\end{document}